\documentclass{amsart}

\usepackage{amsmath,amssymb,amsthm,mathtools}
\usepackage{graphicx}
\usepackage[hidelinks]{hyperref}
\usepackage{color}

\newcommand{\qua}{\hskip 0.4em \ignorespaces}
\def\arxiv#1{\relax\ifhmode\unskip\qua\fi
\href{http://arxiv.org/abs/#1}%
{\tt arXiv:\penalty -100\unskip#1}}

\def\MR#1{\relax\ifhmode\unskip\qua\fi
\href{https://mathscinet.ams.org/mathscinet-getitem?mr=#1}{\tt MR#1}}
\def\ZB#1{\relax\ifhmode\unskip\qua\fi
\href{https://zbmath.org/?q=an:#1}{\tt Zbl\:#1}}
\def\xox#1{\csname xx#1\endcsname}

\renewenvironment{thebibliography}[1]{
  \begin{oldthebibliography}{ABC12}\small
    \setlength{\itemsep}{.5ex}
    \setlength{\parskip}{0em}
}
{
  \end{oldthebibliography}
}

\newcommand{\N}{\mathbb N}
\newcommand{\Z}{\mathbb Z}
\newcommand{\Q}{\mathbb Q}
\newcommand{\Zp}{\Z/p\Z}
\newcommand{\Zr}{\Z/r\Z}

\newcommand{\sZ}{s_{\mathbb Z}}
\newcommand{\lkf}{\lambda}
\newcommand{\co}{\colon}
\newcommand{\Jac}[2]{\left(\frac{#1}{#2}\right)}
\newcommand{\coker}{\operatorname{coker}}
\renewcommand{\epsilon}{\varepsilon}

\newtheorem{theorem}{Theorem}[section]
\newtheorem{lemma}[theorem]{Lemma}
\newtheorem{proposition}[theorem]{Proposition}

\newtheorem{question}[theorem]{Question}
\theoremstyle{definition}
\newtheorem{definition}[theorem]{Definition}

\theoremstyle{remark}
\newtheorem{remark}[theorem]{Remark}
\newtheorem{example}[theorem]{Example}

\numberwithin{equation}{section}
\begin{document}

\title{Lens spaces that cannot be obtained by two integral surgeries}

\author{Marco Golla}
\address{CNRS and Aix-Marseille Universit\'e, Institut de Mathématiques de Marseille (I2M) - UMR 7373, 3 place Victor Hugo, 13331 Marseille, France}
\email{marco.golla@univ-nantes.fr}

\author{Marc Kegel}
\address{Universidad de Sevilla, Avda.\ Reina Mercedes s/n, 41012 Sevilla, Spain}
\email{mkegel@us.es}

\date{\today}

\begin{abstract}
In this article, we demonstrate that there exists an infinite family of lens spaces $L(p,q)$ that cannot be obtained from the $3$-sphere by two integral surgeries. The first such lens space in our family is $L(177,29)$.
\end{abstract}

\keywords{Lens spaces, surgery numbers}

\makeatletter
\@namedef{subjclassname@2020}{\textup{2020} Mathematics Subject Classification}
\makeatother
\subjclass[2020]{57K10; 57K14, 57R65}

\maketitle

\section{Introduction}

For a closed oriented $3$-manifold $Y$ we define the \textit{integral surgery number} $\sZ(Y)$ to be the minimal number of components of a framed link in $S^3$ on which surgery produces $Y$.
This is the integral version of the surgery-number invariants introduced by Auckly~\cite{Auckly}. Computing surgery numbers is notoriously difficult and in general remains poorly understood. For some recent progress, see for example~\cite{HKL,HomLidman,LiuPiccirillo,Azarpendar,DLME,GuoYan}.

Even for lens spaces, determining $\sZ$ appears to be difficult and was explicitly posed by Melvin as Problem~11.2.9 in~\cite{BIRS} (see also~\cite{mathoverflow}). Homology gives elementary necessary obstructions to a lens space being obtainable by a single integral surgery~\cite{FintushelStern}. The complete classification of lens spaces that are obtained from $S^3$ by a single integral surgery is considerably deeper and follows from Greene's resolution of the lens space realization problem~\cite{Greene}. 
On the other hand, it remained open whether there exists a lens space that cannot be obtained from $S^3$ by integral surgery on a $2$-component link. The purpose of this note is to give an infinite family of such lens spaces.

\begin{theorem}\label{thm:main}
There exist infinitely many lens spaces with integral surgery number~$\sZ=3$. More precisely, for each positive integer $a \equiv 7 \pmod 8$, we have
$$\sZ\big(L(3a^2+5a-5,3a+8)\big) =3.$$
\end{theorem}

\subsection*{Lower bounds for integral surgery numbers}

The idea of the proof is to use Donaldson's diagonalisation theorem~\cite{Donaldson-diagonalisation} to obstruct the existence of definite fillings with $b_2 = 2$, see for example~\cite{Lisca2007rationallens,AMP}. This rules out integral surgery presentations on $2$-component definite links. 

We exclude surgery along $2$-component indefinite links by analyzing the underlying number theory of the linking form. Theorem~\ref{thm:criterion} describes an elementary general obstruction for a lens space $L(p,q)$ to arise by surgery along an indefinite $2$-component link. This obstruction, together with the definite obstruction, is enough to prove Theorem~\ref{thm:main}, but it only works if $p\equiv 1 \pmod 8$ (see Proposition~\ref{prop:representability}).

We also give other, less elementary, criteria in Theorem~\ref{thm:criteria2}, that can be used to obstruct surgery along an indefinite $2$-component link. In Section~\ref{s:examples}, we present three different lens spaces with integral surgery number $3$ that are obstructed with Theorem~\ref{thm:criteria2} but cannot be obstructed via the obstruction from Theorem~\ref{thm:criterion}.

Indeed, these criteria apply more generally to all 3-manifolds whose first homology is finite and cyclic.

\begin{remark}
    The definite filling obstruction extends immediately to higher surgery numbers~\cite{AMP}. However, the linking-form obstruction is essentially restricted to rank 2. Indeed, it follows from work of Edmonds~\cite{Edmonds} that every cyclic linking form has an indefinite integral presentation by a $(3\times 3)$-matrix. We refer to Section~\ref{sec:Edmonds} for details.
\end{remark}

\begin{remark}
 We also note that the recent instanton bound $q_3$ for the surgery number from~\cite{DLME} does not detect the examples considered here. In fact, their invariant $q_3$ vanishes on every lens space $L(p,q)$ with $p$ odd.

Indeed, any lens space can be obtained from the unknot by positive or negative surgery coefficients, and since $p$ is odd, $L(p,q)$ is a $(\Z/2\Z)$-homology sphere, so Corollary~2 of~\cite{DLME} applies and gives $q_3(L(p,q))=0$.
\end{remark}

\subsection*{Upper bounds on integral surgery numbers}
Note that many lens spaces can be obtained by integer surgery along a $2$-component link:
\begin{itemize}
    \item Lens spaces bounding rational homology balls~\cite{Lisca2007rationallens, BakerBuck_Lecuona}.
    \item Lens spaces $L(p,q)$ with $q\equiv \pm 1$ $\operatorname{mod}(p)$ and lens spaces obtained by rational surgery on torus knots~\cite{Moser}.
    \item Lens spaces obtained by surgery on a link consisting of two Berge--Gabai knots~\cite{Berge1991solidtorus,Gabai1989solidtorus} (see also the comment by Ian Agol in~\cite{mathoverflow}).
\end{itemize}

Using the above constructions, we performed computer experiments to verify that most lens spaces for small values of $p$ and $q$ can be obtained by two integer surgeries, but the following question remains open.

\begin{question}
Is the integral surgery number of lens spaces universally bounded? For instance, is $\sZ(L(p,q)) \le 3$ for every $p,q$?
\end{question}

\subsection*{Embedding lens spaces in 4-manifolds}

The aforementioned paper of Edmonds~\cite{Edmonds} shows that every lens space \emph{locally-flatly} embeds in $\#^2(\mathbb{CP}{}^2\# \overline{\mathbb{CP}}{}^2)$, in $\#^4 S^2 \times S^2$, and in $\#^8\mathbb{CP}{}^2$.

The last two statements were known not to hold for \emph{smooth} embeddings: the case of $\#^8\mathbb{CP}{}^2$ is a consequence of Donaldson's diagonalisation theorem (as in~\cite{AMP}), and that of $\#^4 S^2 \times S^2$ was shown in~\cite[Theorem~4.14]{AcetoGollaLarson} as a corollary of Furuta's 10/8-Theorem.

A consequence of the techniques of this paper is the following result, which was suggested to us by Livio Ferretti and Brendan Owens.

\begin{theorem}\label{thm:noembedding}
    There are infinitely many lens spaces which do not smoothly embed in $\#^2(\mathbb{CP}{}^2\# \overline{\mathbb{CP}}{}^2)$. For example, $L(177,29)$ is such a lens space.
\end{theorem}

\subsection*{Acknowledgements}
This project was motivated by the Mathoverflow question \cite{mathoverflow}. We started the collaboration and developed the first obstructions during the semester program \textit{Braids} (Feb 1--May 6, 2022) at the Institute for Computational and Experimental Research in Mathematics (ICERM). We thank
ICERM for the invitation and financial support.

We thank Ian Agol for sharing his insight about this problem at~\cite{mathoverflow}. Moreover, we thank Agniva Roy and Hugo Zhou for interesting discussions. They had also independently compiled large databases of lens spaces with integral surgery number 2; comparing our data with theirs revealed an error in an earlier version of our code. We thank Brendan Owens for useful discussions, for pointing us to relevant code in GAP, and for suggesting, together with Livio Ferretti, Theorem~\ref{thm:noembedding}. We also thank Francesco Lin for help with a reference. We acknowledge the use of GAP~\cite{GAP} and Sage~\cite{Sagemath} during the exploratory phase of this project. MK thanks Soheil Azarpendar for useful discussion, in particular for pointing to~\cite{AMP} and for mentioning the potential of the $10/8$-theorem to obstruct indefinite fillings.

\subsection*{Statement on AI use}
MK acknowledges assistance from various AI tools for checking computations and to discuss possible approaches during the development of this project. In particular, MK used ChatGPT Sol in developing and checking an earlier version of the proof of Lemma~\ref{lem:primitive-diagonal}, for explanations and computations of Jacobi symbols and related arithmetic and congruence calculations. 

The proof of Lemma~\ref{lem:primitive-diagonal} appearing in the present manuscript was subsequently revised and improved by the authors. 
All mathematical results and arguments in the final manuscript were established and verified by the authors, who take full responsibility for their correctness.

We also acknowledge that Zhengbo Qu used GPT Astra~6 to find examples of lens spaces with integral surgery number at least 3.

\subsection*{Individual grant support}
MK is supported by a Ram\'on y Cajal grant (RYC2023-043251-I) and by the
project PID2024-157173NB-I00 funded by MCIN/AEI/10.130\-39/501100011033, ESF+
and FEDER, EU; and by a VII Plan Propio de Investigaci\'on y Transferencia
(SOL2025-36103) of the University of Sevilla.

\section{Indefinite surgeries, linking forms, and Jacobi symbols}

In this section, we define an invariant of the linking form of homology lens spaces in terms of Jacobi symbols and study its relation to surgery presentations. 
We start by recalling the definitions and some properties of Jacobi symbols. A standard reference for that is~\cite{NumberTheory}.

\subsection{Jacobi symbols}\label{ss:jacobi}

Let $r$ be an odd prime.
We say that $x$ is a quadratic residue modulo $r$ if there exists $y \in \Zr$ such that $y^2 \equiv x$.
The \emph{Legendre symbol} is the map
\begin{align*}
&\Zr \longrightarrow \{0,\pm 1\}\\
&x \longmapsto \Jac xr := \left\{
\begin{array}{ll}
0 & \text{if $x = 0$},\\
1 & \text{if $x \neq 0$ and is a quadratic residue modulo $r$},\\
-1 & \text{if $x$ is not a quadratic residue modulo $r$}.\\
\end{array}
\right.
\end{align*}
Note that $\Jac xr \equiv x^{\frac{r-1}2} \pmod r$.
The Legendre symbol is multiplicative and it satisfies the quadratic reciprocity law: if $r$ and $s$ are odd primes, then
\[
\Jac sr \Jac rs = (-1)^{\frac{(r-1)(s-1)}4}.
\]

Let $p > 1$ be an odd integer, $p = r_1^{\alpha_1}\cdots r_\ell^{\alpha_\ell}$ be its prime factorisation, and $x$ be an integer.
The \emph{Jacobi symbol} of $x$ modulo $p$ is
\[
\Jac xp = \prod_{i=1}^\ell \Jac x{r_i}^{\alpha_i}.
\]
Unlike the Legendre symbol, the Jacobi symbol does not measure precisely whether a given integer is a quadratic residue modulo $p$. Nevertheless, $\Jac{x^2}p = 1$ for every $x$ coprime with $p$.

Much like the Legendre symbol, however, the Jacobi symbol is multiplicative and satisfies quadratic reciprocity: if $p$ and $q$ are odd integers, then
\[
\Jac qp \Jac pq = (-1)^{\frac{(p-1)(q-1)}4}.
\]
Moreover,
\begin{align*}
\Jac {-1}p &= \left\{
\begin{array}{ll}
1 & \text{if $p \equiv 1 \pmod 4$,}\\
-1 & \text{if $p \equiv 3 \pmod 4$,}
\end{array}
\right.\\
\Jac {2}p &= \left\{
\begin{array}{ll}
1 & \text{if $p \equiv 1,7 \pmod 8$,}\\
-1 & \text{if $p \equiv 3,5 \pmod 8$.}
\end{array}
\right.
\end{align*}

\subsection{The Jacobi symbol of a homology lens space}

Recall that, if $Y$ is a rational homology 3-sphere, the linking form on $Y$ is a non-degenerate symmetric bilinear form
\[
\lambda_Y\co H_1(Y) \otimes H_1(Y) \to \Q/\Z.
\]
(The linking form can be defined for arbitrary closed, oriented 3-manifolds, but we stick to rational homology spheres here for convenience.)
This form can be presented in terms of integral surgery presentations of $Y$, as follows.
Suppose that $Y$ is the boundary of a 4-manifold $W$ with $H_1(W) = 0$ and let $Q$ be a matrix presenting the intersection form of $W$ in some basis. We view $Q$ simultaneously as a bilinear form (on $H_2(W)$) and as an endomorphism (of the same group). For instance, this situation occurs when $Y$ is obtained as surgery along a framed link $L$ in $S^3$, and in this case we can choose $Q$ to be the linking matrix of $L$.

Then $H_1(Y) \cong \coker Q$, and $Q$ induces a $\Q/\Z$-valued bilinear form, presented by $-Q^{-1}$, on $H_1(Y)$. (Note that there are two sign conventions in the literature for the linking form. Here we adopt the one from~\cite[Section 5.3]{GS}.)

With a slight abuse of terminology, we say that a $\Q/\Z$-valued bilinear form $\lambda$ on a finite abelian group $G$ is \emph{presented} by an integral symmetric bilinear form $Q$ on $\Z^m$ if $(\coker Q, -Q^{-1}) \cong (G,\lambda)$ as $\Q/\Z$-valued bilinear forms.

We now focus on homology lens spaces, i.e.\ 3-manifolds whose first homology group is finite and cyclic.
Let $p$ be a positive integer. We denote the $\Q/\Z$-valued bilinear form $\lambda$ on $\Zp$ with $\langle\frac ap\rangle$ if there exist a generator $x \in \Zp$ and an integer $a$ such that $\lambda(x,x) = \frac ap \in \Q/\Z$. We can choose $0 \le a < p$ (although we will not always do so). The form $\lambda$ is non-degenerate if and only if $\gcd(a,p) = 1$.

Very concretely, if we look at lens spaces, the surgery presentation coming from a linear plumbing of spheres shows that the linking form of $L(p,q)$ is presented by the matrix $(q/p)$.
That is to say, there is a generator $x \in H_1(L(p,q))$, which incidentally is the core of one of the solid tori in the genus-1 Heegaard splitting of $L(p,q)$, such that $\lambda(x,x) = \frac qp$.
In other words, $\lambda_{L(p,q)} \cong \langle \frac qp\rangle$.

\begin{lemma}
Let $p$ be a positive odd integer and $\lambda = \langle \frac qp\rangle$ be a linking form on $\Zp$.
The Jacobi symbol $\Jac qp$ is independent of the choice of $q$.
The linking form is non-degenerate if and only if $\Jac qp \neq 0$.
\end{lemma}

\begin{proof}
First off, $\lambda$ is non-degenerate if and only if $\gcd(q,p) = 1$, which in turn is equivalent to $\Jac qp \neq 0$, which proves the last statement.

Now, let $q'$ be another integer such that $\langle \frac{q'}p\rangle \cong \langle \frac qp\rangle$.
This means that there exist two generators $x, x' \in \Zp$ such that $\lambda(x,x) = \frac qp$ and $\lambda(x',x') = \frac{q'}p$.
Since $x$ and $x'$ both generate, there exists $a \in \Zp$, coprime with $p$, such that $x' = ax$.
But then
\[
\frac{q'}{p} = \lambda(x',x') = \lambda(ax,ax) = a^2\lambda(x,x) = \frac{a^2q}{p},
\]
and since $\gcd(a,p) = 1$ we have
\[
\Jac{q'}p = \Jac{a^2q}p = \Jac{a^2}p \Jac qp = \Jac qp,
\]
which proves independence of the presentation.
\end{proof}

\begin{definition}
Let $\lambda \cong \langle \frac qp \rangle$ be a $\Q/\Z$-valued bilinear form on $\Zp$, with $p$ odd.
The \emph{Jacobi symbol} of $\lambda$ is $J(\lambda) = \Jac qp$.

Let $Y$ be a rational homology 3-sphere such that $H_1(Y) \cong \Zp$ for some odd $p$.
The \emph{Jacobi symbol} of $Y$ is $J(Y) = J(\lambda_Y)$.
\end{definition}

\begin{example}
Since $\lambda_{L(p,q)}\cong \langle \frac qp \rangle$, we have $J(L(p,q)) = \Jac qp$.
\end{example}

\subsection{An indefinite rank-\texorpdfstring{$2$}{2} obstruction}

The following result is the main technical result of this article, and will yield to the obstruction for getting a rational homology sphere by surgery on an indefinite $2$-component link.

\begin{proposition}\label{prop:indefinite}
Let $p\equiv1\pmod8$, and $Q$ be a symmetric $(2\times2)$-matrix with integral coefficients such that $\det Q=-p$ and $\coker(Q)\cong\Zp$. Let $\lkf$ be the bilinear form on $\Zp$ presented by $Q$. Then $J(\lkf)=+1$.
\end{proposition}

For the proof of Proposition~\ref{prop:indefinite}, we need a simple algebraic lemma.

\begin{lemma}\label{lem:primitive-diagonal}
Let
\[
 Q=\begin{pmatrix}a&b\\b&c\end{pmatrix}
\]
be an integral symmetric matrix with $\coker(Q)\cong\Z/p\Z$, where $p$ is odd.  After an integral congruence $Q\mapsto P^TQP$ with $P\in\operatorname{SL}(2,\Z)$, one may assume that $\gcd(a,p)=1$.
\end{lemma}

\begin{proof}
It suffices to find a vector $v \in \Z^2$ of the form $v = e_1+te_2$ such that $\gcd(v^t Qv, p) = 1$. If we find such a $v$, then $(v,e_2)$ is a basis of $\Z^2$, and the upper left entry of the intersection matrix in this basis is $v^t Qv= a + 2bt + ct^2$.

Let $r$ be a prime divisor of $p$. We first observe that $a,b,c$ cannot all be divisible by $r$, so that $a + 2bt + ct^2$ is not the zero polynomial modulo $r$.

Indeed, if $r$ divided $a$, $b$, and $c$, then $Q(\mathbb{Z}^2)\subset r\mathbb{Z}^2$. Consequently, reduction modulo $r$ induces a surjection
\begin{equation*}
\coker(Q) = \mathbb{Z}^2/Q(\mathbb{Z}^2) \longrightarrow \mathbb{Z}^2/r\mathbb{Z}^2 \cong (\Zr)^2,
\end{equation*}
which is impossible, since $\coker(Q)\cong \mathbb{Z}_p$ is cyclic, but $(\Zr)^2$ is not.

Since $r$ is an odd prime and $a + 2bt + ct^2$ is a non-zero polynomial, there is an integer $t_r$ such that 
\begin{equation*}
a+2bt_r+ct_r^2\not\equiv 0\pmod r. 
\end{equation*}

Doing this for every prime divisor $r$ of $p$, the Chinese remainder theorem provides an integer $t$ satisfying $t\equiv t_r\pmod r$ for every prime $r\mid p$. Hence 
\begin{equation*}
a+2bt+ct^2\not\equiv 0\pmod r 
\end{equation*}
for every prime divisor $r$ of $p$, and therefore 
\begin{equation*}
\gcd(a+2bt+ct^2,p)=1.
\qedhere
\end{equation*}
\end{proof}

\begin{proof}[Proof of Proposition~\ref{prop:indefinite}]
By Lemma~\ref{lem:primitive-diagonal}, after an integral change of basis, we may write
\[
 Q=\begin{pmatrix}a&b\\b&c\end{pmatrix}
 \qquad\text{with}\qquad \gcd(a,p)=1.
\]

First we compute 
\begin{equation*}
    \lambda(e_2,e_2)=(0, 1) \cdot \big({-Q^{-1}}\big) \cdot \begin{pmatrix}
0 \\
1
\end{pmatrix}=\frac{a}{p}
\end{equation*}
and since $\gcd(a,p)=1$, this element has order $p$. Thus the basic vector $e_2$ represents a generator of $\coker(Q)\cong\Zp$ and therefore
\begin{equation*}
    J(\lambda)=\Jac{a}{p},
\end{equation*}
which we will now compute to be $+1$. The determinant condition of $Q$ gives us $-p=\det(Q)=ac-b^2$, so $b^2\equiv p\pmod a$.
Thus for every odd prime factor $r$ of $a$, we see that $p$ is a non-zero square modulo $r$, i.e.\ 
\begin{equation*}
    \Jac{p}{r}=1.
\end{equation*}
Because $p\equiv1\pmod4$, quadratic reciprocity implies
\begin{equation*}
\Jac{r}{p}=1.
\end{equation*}
In addition, the assumption $p\equiv1\pmod8$ implies that
\[
 \Jac{-1}{p}=1,\qquad \Jac{2}{p}=1.
\]
In summary, we can factor $a=\pm1 \cdot 2^k \cdot r_1\cdots r_n$ into its sign, its powers of $2$, and its odd prime factors $r_i$ and get
\begin{equation*}
 \Jac{a}{p}=\Jac{\pm1}{p}\cdot\Jac{2}{p}^k \cdot \Jac{r_1}{p}\cdots\Jac{r_n}{p}=1,
\end{equation*}
which proves the statement.
\end{proof}

\subsection{A first obstruction to indefinite surgery number 2}

We now use the number-theoretic background from Proposition~\ref{prop:indefinite} to give an obstruction to having indefinite surgery number 2.

\begin{theorem}\label{thm:criterion}
Let $Y$ be a homology lens space with $H_1(Y) \cong \Zp$ for some $p\equiv1\pmod8$, and
\begin{equation*}
J(Y)=-1.
\end{equation*}
Then $Y$ bounds no indefinite topological $4$-manifold $W$ with $H_1(W) = 0$ and $b_2(W) = 2$. In particular, $Y$ has no indefinite surgery presentation along a $2$-component link.
\end{theorem}

\begin{proof}
Suppose that $Y$ bounds such a $W$ and let $Q$ be a matrix representing the intersection form of $W$. Since $H_1(Y)\cong\Zp$ and $W$ is indefinite, we have ${\det Q} = -p$.
Since $\lambda_Y$ is presented by $Q$, Proposition~\ref{prop:indefinite} yields that $J(Y) = +1$, which contradicts the assumption $J(Y)=-1$.
\end{proof}

\begin{remark}
    A different way of looking at the theorem above is the fact that such a $Y$ cannot be a surgery along an indefinite 2-component link in any integer homology 3-sphere.

    Also, note that the statement of Theorem~\ref{thm:criterion} is independent of the chosen sign convention for the linking form, thanks to the assumption $p \equiv 1 \pmod 8$.
\end{remark}

\subsection{Enumerating indefinite binary quadratic forms}\label{ss:classification}

In this section, describe how to enumerate indefinite binary quadratic forms. In the next section, we describe how enumeration can be turned into another obstruction to having an indefinite surgery description along a $2$-component link.

We introduce some notation and recall some known facts about quadratic forms, following the presentation of~\cite[Section 15.3.3]{ConwaySloane}. From now on, we will assume that $d$ is a fixed positive integer, and is \emph{not} a square.

If $Q$ is an indefinite binary quadratic form of determinant $-d < 0$, then it admits a \emph{reduced} matrix representative, that is a $2\times2$ matrix $\begin{pmatrix}a&b\\b&c\end{pmatrix}$ where $0 < b < \sqrt{d} < \min\{b+|a|, b+|c|\}$.

This allows us to list all matrices representing an indefinite binary quadratic form of fixed determinant, but in general, there will be redundancies.
We describe a procedure to produce \emph{all} reduced matrix representatives of such a quadratic form.

Start with a reduced matrix $\begin{pmatrix}a_1&b_1\\b_1&a_2\end{pmatrix}$ of determinant $-d<0$.
Define two sequences $(a_j), (b_j)$ as follows.
For each $j \ge 1$ let $b_{j+1}$ be the largest integer such that
\[
b_{j+1} \equiv -b_j \pmod{a_{j+1}}, \quad b_{j+1}^2 < d.
\]
If no such integer exist, let $b_{j+1}$ be the smallest solution of $b_{j+1} \equiv -b_j \pmod{a_{j+1}}$ in absolute value (we choose the positive solution if $2b_j \equiv 0 \pmod{a_{j+1}}$).
Finally, for $j \ge 1$, let $a_{j+2} = \frac{b_{j+1}^2-d}{a_{j+1}}$.

Then the sequence of matrices $Q_j = \begin{pmatrix}a_j&b_j\\b_j&a_{j+1}\end{pmatrix}$ is eventually periodic, and gives a \emph{cycle of reduced forms}. Two reduced matrices represent isometric quadratic forms if and only if they give rise to the same cycle of reduced forms.

We will shorten the notation by writing
\[
a_1 ^{b_1} a_2 ^{b_2} \dots a_{n}^{b_n}
\]
for the associated cycle. We will also drop signs (which can be reconstructed), as they will not matter in what follows.

The bottom line is that one can enumerate explicitly all cycles of indefinite binary quadratic forms of a given determinant, and this can even be done by hand for, say, determinants up to $1000$ in absolute value.
We implemented the algorithm in Sage.

\subsection{A second obstruction to indefinite surgery number \texorpdfstring{$2$}{2}}

In this section, we give a second obstruction to the existence of an indefinite surgery presentation along a $2$-component link.
While this is algebraic in nature and is based on the linking form, its applicability heavily relies on the explicit classification of binary quadratic forms.

\begin{theorem}\label{thm:criteria2}
Let $Y$ be a homology lens space that bounds an indefinite topological $4$-manifold $S$ with $H_1(S) = 0$ and $b_2(S) = 2$. Then
\begin{itemize}
\item[(0)] $\lambda_Y$ admits a presentation by an indefinite binary quadratic form.
\end{itemize}
Let us now suppose that $W$ is \emph{smooth}.
If $\lambda_Y$ is only presented by \emph{even} indefinite binary quadratic forms and $|H_1(Y)|$ is odd, then we have that
\begin{itemize}
\item[(1)] if $Y = \partial P$, for a spin $4$-manifold $P$, we get
    \begin{itemize}
        \item [(a)] $\sigma(P) \equiv 0 \pmod{16}$,
        \item [(b)] either $\sigma(P) \ge 0$ or $b_2^+(P) \ge 2$, and
        \item [(c)] $b_2^+(P), b_2^-(P) \ge \frac18|\sigma(P)|$;
    \end{itemize}
\item[(2)] if $Y$ is an L-space, $d_{\textup{spin}}(Y)=0$.
\end{itemize}
In particular, all statements apply whenever $Y$ admits an indefinite surgery presentation along a $2$-component link.
\end{theorem}

Note that when $|H_1(Y)|$ is odd, $Y$ has a unique self-conjugate spin$^c$ structure, which also corresponds to the unique spin structure on $Y$.
Here $d_{\text{spin}}(Y)$ denotes the Heegaard Floer correction term of $Y$ in the spin structure.
Heegaard Floer correction terms were introduced by Ozsv\'ath and Szab\'o in~\cite{OSz-absolutelygraded}.

\begin{proof}
Point (0) is obvious, and has in fact already been used in the proof of Theorem~\ref{thm:criterion}.

Let us prove the remaining statements.
Since $\lambda_Y$ can only be presented by even quadratic forms, the intersection form of $S$ is even and hence, since $H_1(S) = 0$, $S$ is spin. Since $S$ is indefinite and $b_2(S) = 2$, the signature of $S$ vanishes.

For point (1), we consider the auxiliary closed smooth 4-manifold $X = P \cup {-S}$.
Since $Y$ has a unique spin structure, $X$ is spin.
By additivity of the signature,
\[
\sigma(X) = \sigma(P) + \sigma(-S) = \sigma(P).
\]
Since $Y$ is a rational homology sphere, moreover,
\[
b_2^{\pm}(X) = b_2^\pm(P) + b_2^\pm(S) = b_2^\pm(P) + 1.
\]
To prove point (1a), we apply Rokhlin's signature theorem to $X$ to obtain
\[
\sigma(P) = \sigma(X) \equiv 0 \pmod{16}.
\]
For point (1b), suppose that $\sigma(P)<0$ and that $b_2^+(P)\le 1$. This implies
\[
\sigma(X) = \sigma(P)<0 \, \textrm{ and }\, 1\le b_2^+(X)=b_2^+(P)+1\le2.
\]
Since $X$ is spin, we can apply Donaldson's Theorem B~\cite{Donaldson-diagonalisation} and get that the intersection form of $X$ is either $\begin{pmatrix}0&1\\1&0\end{pmatrix}$ or $\begin{pmatrix}0&1\\1&0\end{pmatrix} \oplus \begin{pmatrix}0&1\\1&0\end{pmatrix}$. In particular, the signature $\sigma(X)$ of $X$ is vanishing, a contradiction.

To prove point (1c), we apply Furuta's 10/8-theorem~\cite{Furuta-10over8} to $X$ and get
\[
b_2^+(P) + 1 = b_2^+(X) \ge \frac18|\sigma(X)|+1 = \frac18|\sigma(P)|+1.
\]
Finally, to prove point (2), we use Lin's inequalities from Pin(2)-Floer homology~\cite[Theorem~5]{Lin}.
Since $Y$ is an L-space, $\alpha(Y) = \beta(Y) = 2d_{\text{spin}}(Y)$: this follows from~\cite[Proposition~20]{Lin}. Since $p$ is odd, there is a unique spin structure on $Y$, and therefore $\alpha(Y)$ and $\beta(Y)$ are well-defined.
Note that, in particular, for $S^3$ all three invariants vanish~\cite[Section~3.8]{Manolescu}.
Since $Y = \partial S$ is the boundary of a spin 4-manifold with $b_2^+ = b_2^- = 1$, there is a spin cobordism from $S^3$ to $Y$ with $b_2^+ = b_2^- = 1$, and applying Lin's inequalities one gets
\[
2d_{\text{spin}}(Y) = \alpha(Y) \ge \beta(S^3) = 0.
\]
Since $-Y = \partial(-S)$ also has the same properties, we get
\[
-2d_{\text{spin}}(Y) = -\alpha(Y) \ge \beta(S^3) = 0.
\]
Combining the two inequalities, the result follows.
\end{proof}

\begin{remark}
    Point (1c) can be improved slightly by using the refinement of Furuta's inequality due to Hopkins, Lin, Shi, and Xu~\cite{Mahowald}.

    The 11/8-Conjecture posits that $b_2^+(X) \ge \frac3{16}|\sigma(X)|$ (see~\cite{11over8}) for any closed spin 4-manifold $X$. So, conjecturally, one has the stronger inequality
    \[
    b_2^+(P) \ge \frac3{16}|\sigma(P)|-1.
    \]

    We looked for examples of lens spaces that passed the obstruction of Theorem~\ref{thm:criterion} and those of (0), (1a), (1b), and (2) in Theorem~\ref{thm:criteria2}, but could not find any.
    More precisely, we tested all lens spaces $L(p,q)$ with $p$ odd and up to $p \le 10000$, as well as some lens spaces with continued fraction $[a_1,a_2,a_3]^-$, but could not find any.
\end{remark}

\subsection{An excursus, or: how not to use binary quadratic forms}

Proposition~\ref{prop:indefinite} gives an algebraic criterion for the (non-)representability of a given linking form on a cyclic group by an indefinite binary quadratic form, which ends up giving a surgery obstruction in Theorem~\ref{thm:criterion}.

In this section, we want to discuss the limitations of this approach when working with other residue classes modulo 8 and what happens in the definite case.
Most of the arguments are inspired by~\cite[Section~3]{Edmonds}.

Throughout the section, $p$ is a positive integer and $\lambda$ is a linking form on $\Z/p\Z$.

\begin{proposition}\label{prop:representability}
    Suppose that $p$ is odd.
    The form $\lambda$ is represented by both a definite and an indefinite binary quadratic form, unless $J(\lambda) = -1$ and $p \equiv 1 \pmod 8$.
    In this case, $\lambda$ is represented by a definite form, but not by an indefinite one.
\end{proposition}

As we will see in the proof of Proposition~\ref{prop:even}, the assumption that $p$ is odd is essential for the indefinite case.

\begin{proof}
    Let $\lambda = \langle \frac qp\rangle$.
    Following the proof of~\cite[Theorem~3.2]{Edmonds}, the $(2\times 2)$-matrix
    \[
    Q =
    \begin{pmatrix}
        -cp & bp \\
        bp & -a
    \end{pmatrix},
    \]
    whose inverse is
    \[
    \begin{pmatrix}
        \frac a p & b \\
        b & c
    \end{pmatrix},
    \]
    presents the linking form $\langle \frac qp\rangle$ if
    \[
    a \equiv q \pmod p \quad \text{and} \quad |{\det Q}| = p.
    \]
    Moreover, $Q$ is definite if $\det Q = p$ and indefinite if $\det Q = -p$.

    Note that the statement is symmetric if we change sign to $q$: if $Q$ represents $\langle \frac qp \rangle$, then $-Q$ represents $\langle -\frac qp \rangle$, and $Q$ is definite (respectively, indefinite) if and only if $-Q$ is.

    To prove the statement, for each $\epsilon = \pm 1$ we need to find $a$, $b$, and $c$ satisfying
    \begin{equation}\label{eq:edmonds}
    \left\{
    \begin{array}{l}
    a \equiv q \pmod p, \\
    ac-pb^2 = \epsilon.
    \end{array}
    \right.
    \end{equation}

    When $a$ is fixed, the second equation in~\eqref{eq:edmonds} has a solution in $b$ and $c$ if and only if $-pb^2 \equiv \epsilon \pmod a$, or equivalently if $-\epsilon p$ is a square residue modulo $a$.

    Fix $\sigma \in \{-1,1\}$. Since $p$ is odd, by Dirichlet's theorem on primes in arithmetic progressions, there exists a prime $a$ satisfying the congruence in~\eqref{eq:edmonds} and such that $a \equiv \sigma \pmod 4$.

    Let also $\eta \in \{-1,1\}$ be defined by $p \equiv \eta \pmod 4$.
    Then $-\epsilon p$ is a square modulo $a$ if and only if $\Jac{-\epsilon p}a = 1$. By multiplicativity of the Jacobi symbol and quadratic reciprocity, we get
    \[
    \Jac{-\epsilon p}a = \Jac{-\epsilon}a \Jac p a = (-1)^{\frac{p-1}2\cdot \frac{a-1}2}\Jac{-\epsilon}a \Jac q p = \eta^{\frac{\sigma-1}2} \Jac{-\epsilon}a \Jac q p
    \]
    so that
    \[
    \Jac{-\epsilon p}a = 1
    \quad \text{if and only if}\quad 
    \eta^{\frac{\sigma-1}2}\Jac{-\epsilon}a = \Jac q p .
    \]

    The values of $\eta^{\frac{\sigma-1}2}\Jac{-\epsilon}a$ are given in the following table:
    \[
    \begin{array}{c|c|c}
                &   \epsilon = 1    & \epsilon = -1\\
    \hline
    \eta = 1  &   -\sigma           & 1\\
    \hline
    \eta = -1 &   1                 & \sigma
    \end{array}
    \]
    This means that we have a solution of~\eqref{eq:edmonds} with $a$ prime and odd unless $(J(\lambda), \epsilon, \eta) = (-1,-1,1)$ or $(J(\lambda),\epsilon,\eta) = (-1,1,-1)$.

    In the latter case, it suffices to change the sign of $\lambda$, which flips the sign of $J(\lambda)$ since $p \equiv 3 \pmod 4$.

    Let us look at the former case: $(J(\lambda), \epsilon, \eta) = (-1,-1,1)$.
    By Proposition~\ref{prop:indefinite}, we know that if $p \equiv 1 \pmod 8$ then there are indeed no indefinite forms representing $\lambda$.
    This leaves out the case $p \equiv 5 \pmod 8$, for which we \emph{do} find solutions of~\eqref{eq:edmonds}.
    Instead of looking for solutions with $a$ prime, we now look for solutions with $a = 2r$, and $r$ a prime. (Primes $r$ such that $2r \equiv q \pmod p$ exist by Dirichlet's theorem, again.)
    The strategy is the same: a solution to~\eqref{eq:edmonds} is now equivalent to having a solution to the congruences
    \[
    \begin{cases}
        -pb^2 &\equiv -1 \pmod r,\\
        pb^2 &\equiv -1 \pmod 2.
    \end{cases}
    \]
    The second congruence is equivalent to asking that $b$ be odd.
    The first congruence is equivalent to $\Jac{p}{r} = 1$.
    But now:
    \[
    \Jac{p}{r} = \Jac{r}{p} = \Jac{2}{p}\Jac{a}{p} = -J(\lambda) = 1,
    \]
    since $p \equiv 5 \pmod 8$.
\end{proof}

The proposition we just proved says that, if $p \not\equiv 1 \pmod 8$, then $\lambda$ on $\Z/p\Z$ is presented by an indefinite binary form. (Recall that $p$ is assumed to be odd and positive.)

We now turn to the question of whether there exist \emph{even} indefinite binary quadratic forms representing $\lambda$.

\begin{proposition}\label{prop:even}
    When $p$ is even or $p \equiv 3 \pmod 4$, there is no even indefinite binary quadratic form representing $\lambda$.

    When $p \equiv 1 \pmod 8$, there is an even indefinite binary quadratic form representing $\lambda$ if and only if $J(\lambda) = 1$.

    When $p \equiv 5 \pmod 8$, there is an even indefinite binary quadratic form representing $\lambda$.
\end{proposition}

\begin{proof}
    If $Q = \begin{pmatrix}a & b \\ b & c \end{pmatrix}$ represents an even, indefinite quadratic form on $\Z/p\Z$.
    Then $a$ and $c$ are even, and $b$ cannot be even, for otherwise $\coker Q$ would not be cyclic. Then
    \[
    -p = \det Q = ac-b^2 \equiv -b^2 \equiv -1 \pmod 4.
    \]
    This proves the first statement.

    Now, let us suppose that either $p \equiv 5 \pmod 8$ or $J(\lambda) = 1$ and $p \equiv 1 \pmod 8$.
    We argue as in the last part of the proof of Proposition~\ref{prop:representability}: we look for a solution of~\eqref{eq:edmonds} with $a = 2r$ and $r$ prime.
    In fact, the argument above shows that such a solution exists, but gives no control on the parity of $c$.
    However,
    \[
    -1 \equiv -p = \det Q = 2rc - b^2 \equiv 2c - 1 \pmod 4,
    \]
    which in turn implies that $c$, too, is even.
\end{proof}

\section{Definite surgeries and Donaldson's Theorem A}

We briefly recall here how to apply Donaldson's Theorem A~\cite{Donaldson-diagonalisation} to give obstructions to the existence of smooth definite fillings of a given rank.
The following proposition and its variations have been used several times in the past five decades to obstruct the existence of fillings. The specific case of lens spaces was greatly influenced by Lisca's seminal work~\cite{Lisca2007rationallens}.

\begin{proposition}\label{p:donaldson}
Let $P(p,q)$ be the minimal negative-definite plumbing whose boundary is $L(p,q)$, and let $\ell = b_2(P(p,q))$.
If $L(p,q)$ has a smooth positive-definite filling $W$ with $b_2(W) = r$.
Then there is an isometric embedding 
\[
Q_{P(p,q)}\hookrightarrow (\Z^{r + \ell},-I).
\]
\end{proposition}

Recall that the plumbing $P(p,q)$ is determined by the Hirzebruch--Jung continued fraction expansion of $p/q$:
\begin{equation*}
p/q = [a_1,\ldots,a_\ell]^- := a_1-\frac{1}{a_2-\dfrac{1}{\ddots-\dfrac1{a_\ell}}},
\end{equation*}
in the sense that $P(p,q)$ is diffeomorphic to the linear plumbing of $\ell$ disc bundles over 2-spheres, with Euler numbers $-a_1, \dots, -a_\ell$, respectively.
Here $a_1, \dots, a_\ell$ are supposed to be integers, each at least $2$, but any such expansion with $a_i$ integer gives a plumbing (which might not be definite, nor minimal).
Its intersection matrix is tridiagonal, with $-a_1, \dots, -a_\ell$ on the diagonal, and $1$s directly above and below it.

\begin{proof}[Proof sketch of Proposition~\ref{p:donaldson}]
The union $-P(p,q) \cup W$ is a positive-definite, clos\-ed, oriented, smooth 4-manifold $X$.
By Donaldson's Theorem A~\cite{Donaldson-diagonalisation}, the intersection form of $X$ is diagonalisable, and the inclusion $-P(p,q) \hookrightarrow X$ induces the required isometric embedding $H_2(-P(p,q)) \hookrightarrow H_2(X)$.
\end{proof}

\section{Applications}\label{s:examples}

We start with the proof of Theorem~\ref{thm:main}. Recall that this theorem asserts that $\sZ(L(3a^2+5a-5,3a+8)) = 3$ for each $a \equiv 7 \pmod 8$.
In fact, we will start by proving something slightly stronger.

\begin{theorem}\label{thm:nobound}
For each $a \equiv 7 \pmod 8$, $L(3a^2+5a-5, 3a+8)$ does not bound a 4-manifold $W$ with $H_1(W) = 0$ and $b_2(W) = 2$, and $s_\Z(3a^2+5a-5, 3a+8) = 3$.
\end{theorem}

\begin{proof}
We define $L_a := L(3a^2+5a-5,3a+8)$. There are three steps to the proof: first, we will prove that $\sZ(L_a) \le 3$; then we will prove that $L_a$ has no indefinite surgery presentation along a $2$-component link, and finally that it has no definite surgery presentation along a $2$-component link.

First, let us write the Hirzebruch--Jung continued fraction expansions of $\frac{3a^2+5a-5}{3a+8}$ and $\frac{3a^2+5a-5}{3a^2+2a-13}$:
\begin{align*}
\frac{3a^2+5a-5}{3a+8} &= [a, 2^{[a+1]}, 3, 2]^-,\\
\frac{3a^2+5a-5}{3a^2+2a-13} &= [2^{[a-2]}, a+4, 3]^-.
\end{align*}

\vskip 0,2cm

We can re-write the first as
\[
\frac{3a^2+5a-5}{3a+8} = a-1 + \frac{3}{3a+8} = a-1 - \frac1{-a-3-\frac1{-3}} = [a-1,-a-3,-3]^-,
\]
which gives a surgery presentation along a 3-component link, coming from the plumbing corresponding to the continued fraction expansion.
This shows that the integral surgery number $\sZ(L_a) \le 3$.

\vskip 0,2cm

To show that $L_a$ is not an indefinite surgery along a 2-component link, we use Theorem~\ref{thm:criterion}.
Note that, when $a \equiv 7 \pmod 8$, we have $3a^2+5a-5 \equiv 1 \pmod 8$, so the hypothesis of Theorem~\ref{thm:criterion} is indeed satisfied, and $3a+8 \equiv 1 \pmod 4$.
Using quadratic reciprocity and the other properties of the Jacobi symbols, stated in Section~\ref{ss:jacobi}, we have
\begin{align*}
J(L_a) &= \Jac{3a+8}{3a^2+5a-5} = \Jac{3a^2+5a-5}{3a+8} = \Jac{3}{3a+8}\\
&= \Jac{3a+8}3 = \Jac{-1}3 = -1.
\end{align*}

By Theorem~\ref{thm:criterion}, $L_a$ is not an indefinite surgery along a 2-component link.

\vskip 0,2cm

Finally, we prove that $L_a$ does not bound a definite 4-manifold whose second homology has rank 2. If $L(p,q)$ bounds a positive-definite $4$-manifold $W$ with
$b_2(W)=2$, Proposition~\ref{p:donaldson} implies that
\[
Q_{P(p,q)} \hookrightarrow (\mathbb Z^{b_2(P(p,q))+2},-I).
\]
On the other hand, if $L(p,q)$ bounds a negative-definite $4$-manifold
$W$ with $b_2(W)=2$, then $-W$ is a positive-definite filling of $-L(p,q)\cong L(p,p-q)$,
and Proposition~3.1 implies that
\[
Q_{P(p,p-q)} \hookrightarrow
(\mathbb Z^{b_2(P(p,p-q))+2},-I).
\]
Consequently, in order to rule out definite surgery presentations on
two-component links, it suffices to prove that neither 
\begin{align*}
    Q_1 := Q_{P({3a^2+5a-5},{3a+8})} \textrm{ nor } Q_2 := Q_{P({3a^2+5a-5},{3a^2+2a-13})}
\end{align*}
embeds in codimension 2 in a negative definite diagonal lattice.

We start with the former.
Call $v_1, \dots, v_{a+4}$ the basis vectors of $Q_1$ coming from the plumbing, so that $v_1\cdot v_1 = -a, v_{a+3}\cdot v_{a+3} = -3$, $v_j \cdot v_j = -2$ otherwise, and $v_j\cdot v_{j+1} = 1$.
Suppose that $Q_1$ embeds in $\Z^{a+6}$, with orthonormal basis $e_1, \dots, e_{a+6}$.
Since $a > 2$, up to automorphisms of $\Z^{a+6}$, the only possible embedding for the chain of $-2$s is $e_1-e_2, \dots, e_{a+1}-e_{a+2}$, and the only possible embedding for $v_{a+4}$ is $e_{a+3}-e_{a+4}$.
The vector $v_{a+3}$ is now forced to be $e_{a+2}-e_{a+3}-e_{a+5}$.

Finally, the vector $v_1$ cannot use any of the basis vectors $e_2,\dots,e_{a+2}$ (by looking at orthogonality to $v_3, \dots, v_{a+2}$ and $v_1\cdot v_1 = -a > -a-1$), and contains the vector $-e_1$. By orthogonality to $v_{a+3}$ and $v_{a+4}$, it contains a multiple of $e_{a+5}-e_{a+3}-e_{a+4}$.
Summing up, $v_1 = -e_1 + m(e_{a+5}-e_{a+3}-e_{a+4}) + ne_{a+6}$, which implies
\[
a = -v_1\cdot v_1 = 1 + 3m^2 + n^2,
\]
but $a \equiv 7 \pmod 8$ and squares modulo 8 are $0$, $1$, or $4$ and thus a direct inspection yields that the right-hand side modulo 8 is either $0$, $1$, $2$, $4$, $5$, or $6$, but never 7.

We now look at $Q_2$.
Call $w_1, \dots, w_a$ the basis vectors of $Q_2$ coming from the plumbing: $w_{a-1}\cdot w_{a-1} = -a-4$, $w_a\cdot w_a = -3$, $w_j \cdot w_j = -2$ for $j = 1, \dots, a-2$, and $w_j \cdot w_{j+1} = 1$.
Suppose that $Q_2$ embeds in $\Z^{a+2}$, with orthonormal basis $e_1, \dots, e_{a+2}$.
Since $a > 5$, up to automorphisms of $\Z^{a+2}$, the only possible embedding for the chain of $-2$s is $e_1-e_2, \dots, e_{a-2}-e_{a-1}$, and the only possible embedding for $w_a$ is $e_a - e_{a+1} - e_{a+2}$.
Let us look at $w_{a-1}$: using the intersection information with the other basis vectors, we see that $w_{a-1}$ has to be of the form
\[
(m-1)(e_1 + \dots + e_{a-2}) + me_{a-1} + (x+y-1)e_a + xe_{a+1} + ye_{a+2},
\]
but the square of this class is
\[
(a-2)(m-1)^2 + m^2 + (x+y-1)^2 + x^2 + y^2 \equiv m-1 + m + x+y-1 + x+y \equiv 0 \pmod 2,
\]
while $w_{a-1}\cdot w_{a-1} = -a-4 \equiv 1 \pmod 2$.

We have therefore shown that $L_a$ admits no integral surgery presentation along a 2-component link. This also rules out a surgery along a single knot, since any integral surgery presentation on a knot can be stabilized by adding a split $(\pm1)$-framed unknot.
\end{proof}

We now give three examples of lens spaces that have integral surgery number $3$. All three examples have $p\not\equiv 1 \pmod 8$, so Theorem~\ref{thm:criterion} does not apply. Instead, we will use the different obstructions from Theorem~\ref{thm:criteria2}.

\begin{proposition}
    $s_\Z(L(112,15)) = 3$.
\end{proposition}

\begin{proof}
The continued fraction expansion of $112/15$ is $[8,2,8]^-$, so the canonical negative plumbing for $L(112,15)$ gives a surgery presentation along a 3-component link. This proves that $s_\Z(L(112,15)) \le 3$.

\vskip 0,2 cm

Indefinite binary quadratic forms of determinant $-112$ with cyclic cokernel are classified as explained in Section~\ref{ss:classification}.
The resulting cycles are
\[
9 ^2 12 ^{10} 1 ^{10} 12 ^2 9 ^7 7 ^7, \quad 3 ^8 16 ^8 3 ^{10} 4 ^{10},
\]
and the linking form they present are $\langle\pm \frac1{112}\rangle, \langle\pm \frac3{112}\rangle$. We claim that $\langle\frac{15}{112}\rangle$ is not isometric to any of these four.
To see this, we need to show that the congruences
\[
15 = \pm k^2 \pmod{112}, \qquad 15 = \pm 3 k^2 \pmod{112}
\]
have no solution.

We start with the second congruence: reducing modulo $8$, we have $15 \equiv \pm 3k^2 \pmod 8$, which is impossible since the right-hand side is $0$, $3$, $4$, or $5$ modulo $8$.

Let us look at the first congruence.
By the Chinese remainder theorem, we need to show that for $\epsilon = -1,+1$ there is no solution to the system:
\[
\left\{
\begin{array}{l}
    15 \equiv \epsilon k^2 \pmod 8\\
    15 \equiv \epsilon k^2 \pmod 7.
\end{array}
\right.
\]
The first congruence has a solution if and only if $\epsilon = -1$, while the second has a solution if and only if $\epsilon = +1$.

Hence, by Theorem~\ref{thm:criteria2}(0), $L(112,15)$ does not have an indefinite surgery presentation along a 2-component link.

\vskip 0,2cm

To show that $L(112,15)$ has no definite surgery presentation along a $2$-component link either, we again use Proposition~\ref{p:donaldson}.
As mentioned above, $112/15 = [8,2,8]^-$.
We also have $112/97 = [2^{[6]},4,2^{[6]}]^-$.

We will argue by contradiction that neither of the intersection forms of the two plumbings has a codimension-2 embedding in a diagonal lattice.

For the first, we are looking at an embedding in $\Z^5$, with orthonormal basis $e_1, \dots, e_5$.
Call $v$ and $w$ the two vectors of square $8$.
Since they both intersect the vectors of square $2$ once, they are primitive.
The only ways to write $8$ as a square of at most five integers is $8 = 4 + 4 = 4 + 1 + 1 + 1 + 1$.
The first is not allowed, since the corresponding vector would be primitive.
In the second case, there are two possibilities: either $v$ and $w$ have the coefficient $2$ in front of two different $e_i$s, or they have it in front of the same.
The former case is impossible, since then
\[
0 = v\cdot w \equiv 2 + 2 + 1 + 1 + 1 \equiv 1 \pmod 2.
\]
In the second case, up to automorphisms of $\Z^5$ we can write:
\[
v = 2e_1 + e_2 + \dots + e_5, \quad w = \pm 2e_1 \mp (e_2 + \dots + e_5).
\]
(The signs for $w$ are opposite because $v \cdot w = 0$.)
But then the vector of square 2 has to be of the form $\pm e_1 \pm e_2$, which is easily seen to be incompatible with the constraints.

For the other embedding, we actually see that the corresponding lattice has only one embedding (up to automorphism) in $\Z^n$ for $n \ge 16$, and that embedding uses all sixteen basis vectors.
Indeed, the two chains of $2$s are forced, and use fourteen distinct basis vectors, and the vector of square $4$ necessarily uses a vector from the left chain, one from the right chain, and two extra vectors.
\end{proof}

\begin{proposition}
    $s_\Z(L(213,22)) = 3$.
\end{proposition}

\begin{proof}
We have $213/22 = [10, 4, 2^{[6]}]^- = [10,3,-7]^-$, so the associated plumbing gives a surgery presentation of $L(213,22)$ along a 3-component link, and this shows that $s_\Z \le 3$.

\vskip 0,2cm

We start by showing that there is no indefinite presentation along a 2-component link.
As in the previous proof, we classify all indefinite binary quadratic forms of determinant $-213$:
\[
12 ^3 17 ^{14} 1 ^{14} 17 ^3 12 ^9 11 ^{13} 4 ^{11} 23 ^{12} 3 ^{12} 23 ^{11} 4 ^{13} 11 ^9, \quad
6 ^9 22 ^{13} 2 ^{13} 22 ^9.
\]
Note that $213 = 3\cdot 71$, so the condition on the cokernel is automatically satisfied.

The first one represents the linking forms $\langle\pm \frac1{213}\rangle$, but $22 \neq \pm k^2 \pmod {213}$, because
\[
\Jac{22}{71} = \Jac{2}{71}\Jac{11}{71} = \Jac{71}{11} = \Jac{5}{11} = \Jac{3}{5} = -1,
\]
and as $71 \equiv 3 \pmod 4$, $22$ is neither a square, nor $-1$ times a square modulo 71 (and hence modulo 213).

So we have only even intersection forms (which actually present the right linking form, but we don't need to check it).

Now, $213/22 = [10, 4, 2^{[6]}]^-$, so the corresponding plumbing is definite and has signature $-8$. This contradicts Theorem~\ref{thm:criteria2}(1a).

\vskip 0,2cm

The arguments for showing that there are no definite surgery presentations along a 2-component link make use of Proposition~\ref{p:donaldson} and follow the same line of ideas as above.
We leave the details to the reader.
\end{proof}

\begin{proposition}
    $s_\Z(L(197,13)) = 3$.
\end{proposition}

\begin{proof}
We compute the continued fraction expansion as
\[
197/(197-13) = [2^{[14]}, 8, 2]^- = [-15,7,2]^-,
\]
so the corresponding plumbing gives $s_\Z \le 3$.

\vskip 0,2cm

We now show that there are no indefinite surgery presentations along 2-component links.
First, we classify the indefinite binary quadratic forms of determinant $-197$ (which is a prime):
\[
14 ^1 14 ^{13} 2 ^{13}, \quad 7 ^8 19 ^ {11} 4 ^{13}, \quad 19 ^8 7 ^{13} 4 ^{11}, \quad 1 ^{14}.
\]
Since $197 \equiv 5 \pmod 8$, $\Jac{-1}{197} = 1$. So the three odd intersection forms represent $\pm1$ modulo $197$, but
\[
\Jac{13}{197} = \Jac{5}{13} = \Jac{3}{5} = -1,
\]
so the only forms that can appear are even. (That the last form represents $\pm1$ is obvious. That the other two do, follows from the fact that they both represent $\pm 4 = \pm 2^2$.)

But now we have $197/184 = [2^{[14]}, 8, 2]^-$, so the corresponding plumbing is definite and even, of rank $16$. This contradicts Theorem~\ref{thm:criteria2}(1b).

\vskip 0,2cm

The arguments for showing that there are no definite surgery presentations along a 2-component link make use of Proposition~\ref{p:donaldson} and follow the same line of ideas as above.
We leave the details to the reader.
\end{proof}

\begin{proposition}
    $s_\Z(L(397,39)) = 3$.
\end{proposition}

\begin{proof}
We start by computing the continued fraction expansion:
\[
397/39 = [11,2^{[4]},3,4]^- = [10,-5,2,4]^-,
\]
so $L(397,39)$ is obtained as surgery along a 2-component link in $L(7,4)$ (since $7/4 = [2,4]^-$).
However, $L(7,4)$ is obtained by $(-7)$-surgery on the torus knot $-T(2,3)$ by~\cite{Moser}, and hence $s_\Z(L(397,39)) \le 3$.

\vskip 0,2cm

We now turn to indefinite surgery presentations.
We start by classifying indefinite binary quadratic forms of determinant $-397$ (which is prime):
\[
19 ^6 19 ^{13} 12 ^{11} 23 ^{12} 11 ^{10} 27 ^{17} 4 ^{19} 9 ^{17} 12 ^{19} 3 ^{17} 36 ^{19} 1, \quad 6 ^{17} 18 ^{19} 2.
\]
Here, we use reflection symmetry to abbreviate the cycles; the corresponding full cycles are longer. The first one represents $1$ as a square, so its Jacobi symbol is 1, but $397$ is a prime congruent to 1 modulo 4, so we get
\[
\Jac{39}{397} = \Jac{3}{397} \Jac{13}{397} = \Jac{1}{3} \Jac{7}{13} = \Jac{-1}{7} = -1.
\]
So we are left with the second one, which is even.

Note that neither (1a) nor (1b) in Theorem~\ref{thm:criteria2} obviously applies here.
However, a direct computation using Ozsv\'ath and Szab\'o's recursive formula for correction terms of lens spaces\footnote{Note that Oszv\'ath and Szab\'o use the opposite orientation convention for lens spaces, so their $-L(p,q)$ is our $L(p,q)$, which is $-p/q$-surgery along the unknot.}, and the fact that the spin structure on $L(p,q)$ is indexed by either $\frac{q-1}2$ or $\frac{p+q-1}2$, we have
\[
d_{\textup{spin}}(L(397,13)) = -6.
\]
Thus, Theorem~\ref{thm:criteria2}(2) applies to rule out an indefinite integral $2$-component surgery description.

\vskip 0,2cm

The arguments for showing that there are no definite surgery presentations along a 2-component link make use of Proposition~\ref{p:donaldson} and follow the same line of ideas as above. Again, we do not present the details of these arguments.
\end{proof}

\section{Smooth embeddings in \texorpdfstring{$\#^2(\mathbb{CP}{}^2\# \overline{\mathbb{CP}}{}^2)$}{small connected sums of complex projective planes}}

In this section, we prove Theorem~\ref{thm:noembedding}.
We start with a standard lemma that is well-known to experts. We include it for completeness.

\begin{lemma}\label{lem:splitting}
    Suppose that $Y$ is a rational homology $3$-sphere such that $H_1(Y)$ has square-free order.
    If $Y$ locally-flatly embeds in a closed, oriented $4$-manifold $X$ with $H_1(X) = 0$, then it separates $X$ into two pieces, each with trivial $H_1$.
\end{lemma}

\begin{proof}
    We start by proving that $Y$ separates.
    Since $Y$ and $X$ are both orientable and $Y$ is locally-flatly embedded, $Y$ has a tubular neighbourhood which is homeomorphic to $[-1,1]\times Y$.
    If $Y$ did not disconnect $X$, there would exist a loop in $X \setminus Y$ connecting $(-1,y)$ to $(1,y)$ for some $y \in Y$. Concatenating this loop with the arc $[-1,1]\times\{y\}$ gives a geometric dual to $[Y] \in H_3(X)$, which contradicts the assumption that $H_1(X) = 0$.

    Let now $X_1, X_2$ denote the closures of the two connected components into which $Y$ splits $X$.
    It suffices to show that $H_1(X_1) = 0$.
    First, from the Mayer--Vietoris long exact sequence associated to the decomposition of $X$ as $X_1 \cup X_2$, we extract
    \[
    H_1(Y) \to H_1(X_1) \oplus H_1(X_2) \to H_1(X) = 0,
    \]
    so $H_1(X_1)$ is finite.
    We also see that
    \[
    0 = H_2(Y) \to H_2(X_1) \oplus H_2(X_2) \to H_2(X),
    \]
    so $H_2(X_1)$ is free, since $H_2(X)$ is.
    Since $H_1(X_1)$ is torsion, by the universal coefficient theorem $H_1(X_1) = 0$ if and only if $H^2(X_1)$ is torsion-free, and the latter group is isomorphic to $H_2(X_1,Y)$ by Poincar\'e--Lefschetz duality.
    
    From the long exact sequence of the pair $(X_1,Y)$, together with the observations we have made so far, we get:
    \[
    0 = H_2(Y) \to H_2(X_1) \stackrel{\alpha}{\longrightarrow} H_2(X_1,Y) \to H_1(Y) \to H_1(X_1) \to H_1(X_1,Y) \to 0
    \]
    Note that $H_1(X_1,Y)$ is trivial, since it is finite (by the long exact sequence) and isomorphic to the torsion in $H^2(X_1,Y) \cong H_2(X_1)$.
    Moreover, the torsion in $H_2(X_1,Y)$ is isomorphic to $H_1(X_1)$ by the same argument, and the image of $\alpha$ is non-torsion.
    All in all, we get the short exact sequence
    \[
    0 \to \coker \alpha \to H_1(Y) \to H_1(X_1) \to 0
    \]
    where $\coker\alpha$ contains a subgroup isomorphic to $H_1(X_1)$, and hence its order is divisible by $|H_1(X_1)|$.
    Thus $|H_1(X_1)|^2$ divides $|{\coker\alpha}|\cdot |H_1(X_1)| = |H_1(Y)|$. Since $H_1(Y)$ is square-free, it follows that $H_1(X_1) = 0$.
\end{proof}

Next, we need a number-theoretic lemma, which was proven in much greater generality by Nagel~\cite{Nagel}.

\begin{lemma}\label{lem:number_theory}
    $f(n) = 192n^2+376n+177$ is square-free for infinitely many values of $n \in \N$.
\end{lemma}

\begin{proof}
    We will estimate for how many $n\in\N$ the number $f(n)$ is divisible by the square of a prime. We do this for each prime $p$ separately.
    \begin{itemize}
        \item For $p=2$, we compute $f(n)\equiv177\equiv1 \pmod 2$, and thus $f(n)$ is odd. In particular, it is not divisible by $2^2$.
        \item For $p=3$, we have $f(n)\equiv 376n \equiv n \pmod 3$. So if $3$ divides $f(n)$ then $n\equiv0\pmod 3$. If we write $n=3k$, we compute 
        \begin{align*}
            f(3k)=1728k^2+1128k+177\equiv 3k+6 = 3(k+2) \pmod 9.
        \end{align*}
        Therefore $9$ divides $f(n)$ if and only if $k\equiv 1 \pmod 3$ if and only if $n\equiv 3 \mod 9$. Thus, exactly $1/9$ of all natural numbers $n$ have the property that $9$ divides $f(n)$.
        \item Next, we consider the case $p=5$ or $p=17$, where we claim that $p^2$ never divides $f(n)$. Indeed, if there were an integer $n$ such that $p^2$ divides $f(n)$, then we have
        \begin{align*}
            0\equiv12f(n)=(48n+47)^2-85 \pmod {p^2},
        \end{align*}
        which implies
        \begin{align*}
            0\not\equiv 5\cdot 17=85\equiv(48n+47)^2 \pmod {p^2}.
        \end{align*}
        Since $p$ is either $5$ or $17$, we have that $(48n+47)^2 \equiv 0 \pmod p$, and thus also $(48n+47) \equiv 0 \pmod p$. But this yields the contradiction $(48n+47)^2 \equiv 0 \pmod {p^2}$.
        \item For any other prime $p\notin \{2,3,5,17\}$, we first claim that $f$ has at most two roots modulo $p^2$. Indeed, since $p\neq 2,3$, the leading coefficient $192=3\cdot2^{6}$ is non-zero modulo $p$, and thus the reduction of $f$ modulo $p$ is a quadratic polynomial, so it has at most two roots modulo $p$. Next, we show that every root of $f$ modulo $p$ lifts to a unique root modulo $p^2$. 
        Every lift modulo $p^2$ of a root $\bar n \in \Z$ modulo $p$, is of the form $n = \bar n + kp$ for some $k$. Then we compute
    \[
    f(n) \equiv f(\bar n) + (384\bar n+376)pk \pmod{p^2},
    \]
    which vanishes if and only if
    \[
    (384\bar n+376)k \equiv -\frac{f(\bar n)}p \pmod p.
    \]
    Since $p$ does not divide the discriminant 
    $$\Delta=376^2-4\cdot192\cdot177=5440=2^6\cdot5\cdot17,$$ we have that $\bar n$ is not a double root of $f$. Thus, 
    $$384\bar n + 376 = f'(\bar n) \neq 0 \mod p,$$ which means that every root $\bar n$ uniquely determines a $k$ and therefore $\bar n$ has a unique lift.
    Thus, at most $2/p^2$ of all natural numbers $n$ have the property that $p^2$ divides $f(n)$.
    \end{itemize}
    
    Now we can estimate the density $\rho$ of the values of $n$ for which $f(n)$ is square-free as follows
    \[
    \rho \ge 1 - \frac19 - \sum_{p \neq 2,3,5,17} \frac2{p^2} > \frac{8}{9} - 2\sum_{n \ge 3} \frac1{n^2} = \frac{8}{9} - 2\left(\frac{\pi^2}6-1-\frac{1}{4}\right) > \frac{8}{9} -\frac{8}{10} > 0,
    \]
    so the density of the values of $n$ for which $f(n)$ is square-free is positive, and hence there are infinitely many such values $n$.
\end{proof}

We can now prove Theorem~\ref{thm:noembedding} from the introduction, asserting that there are infinitely many lens spaces that do not embed in $\#^2(\mathbb{CP}{}^2\# \overline{\mathbb{CP}}{}^2)$.

\begin{proof}[Proof of Theorem~\ref{thm:noembedding}]
    First, we show that there are infinitely many values of $a \equiv 7 \pmod 8$ such that $3a^2+5a-5$ is square-free. For that, we write $a=7+8n$ and compute that $3(8n+7)^2 + 5(8n+7) - 5=192n^2+376n+177=f(n)$. By Lemma~\ref{lem:number_theory}, there exist infinitely many values for $n$ such that $f(n)$ is square-free. For instance, $a = 7$ works, yielding $L(177,29)$.

    Pick such a value, and suppose that $L_a := L(3a^2+5a-5,3a+8)$ smoothly embeds in $\#^2(\mathbb{CP}{}^2\# \overline{\mathbb{CP}}{}^2)$.
    By Lemma~\ref{lem:splitting}, $L_a$ splits $\#^2(\mathbb{CP}{}^2\# \overline{\mathbb{CP}}{}^2)$ into two pieces, each of which has trivial $H_1$.
    From the Mayer--Vietoris long exact sequence, at least one of them has $b_2 \le 2$.
    Up to taking a connected sum with $\mathbb{CP}{}^2$ or $\#^2\mathbb{CP}{}^2$ if necessary, $L_a$ bounds a smooth 4-manifold $W$ with $H_1(W) = 0$ and $b_2(W) = 2$, but this contradicts Theorem~\ref{thm:nobound}.
\end{proof}

\section{Ternary quadratic forms cannot obstruct surgery number \texorpdfstring{$3$}{2}}\label{sec:Edmonds}

In this section, we discuss the limitations of our approach for showing that the surgery number has to be at least $3$.

We start off with ternary forms.
First, note that Lin's inequality on spin fillings of 3-manifolds coming from Pin(2)-Floer homology does have a generalisation to the case $b_2^+ = 2$.
So, in principle, one could hope to use the same approach as in Theorems~\ref{thm:criterion} and~\ref{thm:criteria2}.
The argument would need to go like this. Show that there exists a lens space $L(p,q)$ with $p$ odd such that
\begin{itemize}
    \item $\lambda_{L(p,q)}$ is not represented by an odd indefinite ternary form;
    \item say, $d_\textup{spin}(L(p,q))$ is large; and
    \item $L(p,q)$ does not bound a smooth, definite 4-manifold $Z$ with $b_2(Z) = 3$.
\end{itemize}

The last point is rather easy to achieve using Proposition~\ref{p:donaldson} (see, for instance, \cite{AMP}).
The second point, too, is, and even in conjunction with the last one.
The first point, however, \emph{cannot} be achieved, as the following proposition shows.

\begin{proposition}
    For any $p$ and any linking form $\lambda$ on $\Zp$, $\lambda$ is represented by an odd, indefinite ternary quadratic form.
\end{proposition}

\begin{proof}
    According to~\cite[Theorem 3.2]{Edmonds}, $\lambda$ is represented by a binary quadratic form $Q$ (without any control on its definiteness or parity).
    If $Q$ is positive definite, let $Q' = Q \oplus \langle -1 \rangle$, otherwise let $Q' = Q \oplus \langle +1 \rangle$.
    In either case, $Q'$ is an odd, indefinite ternary quadratic form representing $\lambda$.
\end{proof}

\bibliographystyle{myamsalpha}
\bibliography{main}

\end{document}